\documentclass[reqno,12pt]{amsart}

\usepackage{amsmath, amssymb, eucal, amscd, amstext,  enumerate}
\theoremstyle{plain}
\newtheorem{theorem}{Theorem}[section]
\newtheorem{proposition}[theorem]{Proposition}
\newtheorem{lemma}[theorem]{Lemma}
\newtheorem{corollary}[theorem]{Corollary}

\theoremstyle{definition}

\theoremstyle{remark}

\newtheorem{examples}[theorem]{Examples}
\newtheorem{example}[theorem]{Example}

\numberwithin{equation}{section}

\DeclareSymbolFont{AMSb}{U}{msb}{m}{n}

\DeclareMathSymbol{\A}{\mathbin}{AMSb}{"41}
\DeclareMathSymbol{\B}{\mathbin}{AMSb}{"42}
\DeclareMathSymbol{\C}{\mathbin}{AMSb}{"43}
\DeclareMathSymbol{\D}{\mathbin}{AMSb}{"44}
\DeclareMathSymbol{\E}{\mathbin}{AMSb}{"45}
\DeclareMathSymbol{\F}{\mathbin}{AMSb}{"46}
\DeclareMathSymbol{\G}{\mathbin}{AMSb}{"47}

\DeclareMathSymbol{\HH}{\mathbin}{AMSb}{"48}

\DeclareMathSymbol{\I}{\mathbin}{AMSb}{"49}

\DeclareMathSymbol{\N}{\mathbin}{AMSb}{"4E}

\DeclareMathSymbol{\PP}{\mathbin}{AMSb}{"50}

\DeclareMathSymbol{\Q}{\mathbin}{AMSb}{"51}
\DeclareMathSymbol{\R}{\mathbin}{AMSb}{"52}

\DeclareMathSymbol{\SSS}{\mathbin}{AMSb}{"53}

\DeclareMathSymbol{\T}{\mathbin}{AMSb}{"54}
\DeclareMathSymbol{\U}{\mathbin}{AMSb}{"55}
\DeclareMathSymbol{\V}{\mathbin}{AMSb}{"56}
\DeclareMathSymbol{\W}{\mathbin}{AMSb}{"57}
\DeclareMathSymbol{\X}{\mathbin}{AMSb}{"58}
\DeclareMathSymbol{\Y}{\mathbin}{AMSb}{"59}
\DeclareMathSymbol{\Z}{\mathbin}{AMSb}{"5A}

\def\th#1{\bigskip\noindent{\bf #1}\bgroup\it}
\def\tth#1{\bigskip\noindent{\bf #1}\bgroup}
\def\endth{\egroup\par\bigskip}
\def\endtth{\egroup\par\bigskip}

\newcommand{\coz}{\operatorname{coz}}

\newcommand{\vp}{\varphi}
\newcommand{\cH}{\mathcal H}

\begin{document}

\title
[Holomorphic maps between C*-algebras]
{Orthogonally additive holomorphic maps between C*-algebras}

\author[Q. Bu]{Qingying Bu}
\address[Qingying Bu]{Department of Mathematics, University of Mississippi, University, MS 38677, USA}
\email{qbu@olemiss.edu}

\author[M.-H. Hsu]{Ming-Hsiu Hsu}
\address[Ming-Hsiu Hsu]{Department of  Mathematics,
National Central University, Chung-Li,32054, Taiwan.}
\email{hsumh@math.ncu.edu.tw}

\author[N.-C. Wong]{Ngai-Ching Wong}
\address[Ngai-Ching Wong]{Department of Applied Mathematics,
National Sun Yat-sen University, Kaohsiung, 80424, Taiwan.}

\email{wong@math.nsysu.edu.tw}

\thanks{This research is supported partially by the Taiwan NSC grants 099-2811-M-110-015
and  102-2115-M-110-002-MY2}

\date{August 3, 2014; March 20, 2015}



\begin{abstract}
Let $A,B$ be C*-algebras,  $B_A(0;r)$ the open ball in $A$ centered at $0$ with radius $r>0$, and
$H:B_A(0;r)\to B$ an orthogonally additive
holomorphic map.
If $H$ is zero product preserving on positive elements in $B_A(0;r)$,
we show, in the commutative case when  $A=C_0(X)$ and $B=C_0(Y)$, that
there exist weight functions
$h_n$'s and a symbol map $\varphi: Y\to X$ such that
$$
H(f)=\sum_{n\geq1} h_n (f\circ\varphi)^n, \quad\forall f\in B_{C_0(X)}(0;r).
$$
In the general case,
we show that if $H$ is also conformal then
there exist central  multipliers $h_n$'s of $B$ and a surjective Jordan
isomorphism $J: A\to B$ such that
$$
H(a) = \sum_{n\geq1} h_n J(a)^n, \quad\forall a\in B_A(0;r).
$$
If, in addition, $H$ is zero product preserving on the whole $B_A(0;r)$, then $J$ is an algebra isomorphism.
\end{abstract}

\maketitle

\section{Introduction}

Let $E$ and $F$ be (complex) Banach spaces, and $n$ a positive integer. A map $P: E \to F$ is called a \emph{bounded
$n$-homogeneous polynomial} if there is a  bounded  symmetric  $n$-linear
operator $L: E\times\cdots\times E \to F$ such that
$$
P(x) = L(x,\ldots, x), \quad\forall x \in E.
$$
Let  $B_{E}(a;r)$ denote the open ball of $E$ centered at $a$ with radius $r>0$.
A map $H: U\to F$ is said to be \emph{holomorphic} on a nonempty open subset $U$ of $E$
if for each $a $ in $U$ there exist an open ball $B_E(a;r) \subset U$ and
a unique sequence of bounded $n$-homogeneous polynomials $P_n:E \to F$ such that
$$
H(x) = \sum_{n=0}^\infty P_n(x - a)
$$
uniformly for all $x $ in $B_E(a;r)$.
Here, $P_0$ is the constant function with value $H(a)$.
After translation, we can assume $a=0$, and  a holomorphic function $H: B_E(0;r)\to F$ has its Taylor series at zero:
$$
H(x) = \sum_{n=0}^\infty P_n(x)
$$
uniformly for all $x$ in $B_E(0;r)$.  In this case, $P_n$ is given by the vector-valued integration as
\begin{align*}
P_n(x) = \frac{1}{2\pi i}\int_{|\lambda| = 1} \frac{H(\lambda x)}{\lambda^{n+1}}\,d\lambda, \qquad n = 0, 1, 2, \ldots.
\end{align*}
See, for example, \cite[pp.\ 40--47]{Mu}.
For the general theory of homogeneous polynomials and holomorphic functions, we refer to \cite{Mu, Di}.

When $E, F$ are function spaces or Banach algebras,
a  map $\Phi:E\to F$  is said to be {\it
orthogonally additive} if
$$
fg = gf=0\quad\text{implies}\quad \Phi(f + g) = \Phi(f) + \Phi(g),\quad \forall f,g\in E,
$$
and  \emph{zero product preserving} if
$$
fg = 0\quad\text{implies}\quad \Phi(f)\Phi(g) = 0,\quad \forall f,g\in E.
$$
The notions of orthogonally additive and zero product preserving transformations
have been studied by many authors, for example, \cite{Ar83,Ja90,JW96,HS99,CKLW03,araujo04,KLW04,PV,BLL06,CLZ06,PPV08,LNW12,LW13}.


The main goal of this paper is to establish  Theorems \ref{thm:BSTPdphf} and
\ref{thm:Conformal}.

Let $A, B$ be C*-algebras and $H: B_A(0;r)\to B$ an orthogonally additive and zero product preserving holomorphic map.
In the commutative case when $A=C_0(X)$ and $B=C_0(Y)$, the algebras of continuous complex-valued functions vanishing at infinity,
we show in Theorem \ref{thm:BSTPdphf} that
$$
H(f)(y) = \sum_{n\geq 1} h_n(y)f(\varphi(y))^n, \quad\forall y\in Y, \forall f\in B_{C_0(X)}(0;r).
$$
Here,
$\varphi: Y\rightarrow X$ is continuous  wherever some weight function $h_n$ is nonvanishing.

When $A,B$ are general C*-algebras, we assume further that $H$ is conformal, i.e., the derivative $P_1$ of $H$ at $0$ is a bounded
invertible linear map from $A$ onto $B$.  We show in Theorem \ref{thm:Conformal} that
there is a sequence $\{h_n\}$ in the center of the multiplier algebra $M(B)$ of $B$ and
an algebra isomorphism $J:A\to B$ such that
$$
H(a) = \sum_{n\geq1} h_n J(a)^n, \quad\forall a\in B_A(0;r).
$$
If we assume $H$ zero product preserving only on the positive part $A_+$ of $A$ intersecting
$B_A(0;r)$, then $J$ is still a Jordan isomorphism.

To achieve Theorems \ref{thm:BSTPdphf} and \ref{thm:Conformal}, we need to study
homogeneous polynomials first.
In  \cite{Su91}, Sundaresan
characterized the linearization of orthogonally additive $n$-homogeneous
polynomials on $L_p$-spaces.   Several authors have  extended
his results to, e.g., $C(K)$-spaces \cite{PV},
C$^\ast$-algebras \cite{PPV08}, and Banach lattices \cite{BLL06, BB12}.
Adopting these important linearization tools (see Section \ref{s:linear}), in this paper
we  establish Banach-Stone type theorems for orthogonally additive
homogeneous polynomials and holomorphic maps between C*-algebras in Sections \ref{s:zp} and \ref{s:isometry}.

A counterpart of   Theorem \ref{thm:Conformal} is given for orthogonally additive
and zero product preserving holomorphic functions between matrix algebras in \cite{BLW14}.  See Theorem \ref{thm:main-matrix}
and Example \ref{ex:trival-multiplication}  for details.
In a very interesting recent paper \cite{GPPI13}, the authors there consider
orthogonally additive holomorphic maps $H$ between general C*-algebras, which  preserve doubly orthogonality, i.e.,
$$
a^*b = ab^* =0 \quad\text{implies}\quad H(a)^*H(b) = H(a)H(b)^* =0.
$$
In this case, with the extra assumption that the range of $H$ contains an invertible element, a corresponding result to   Theorem \ref{thm:Conformal} is established in \cite{GPPI13} through the technique of JB*-algebras.
Nevertheless, we will see in Example \ref{ex:trival-multiplication} that one cannot directly make use of this new result to study zero product preserving holomorphic maps.

In Theorems \ref{thm:BSTPiso} and \ref{thm:C-isom}, we establish similar representation results for orthogonally additive homogeneous $n$-polynomials,
$P: A\to B$ between C*-algebras, which are  \emph{$n$-isometry} of positive elements, i.e.,
$$
\|P(a)\| = \|a\|^n,\quad\forall a\in A_+.
$$

To end the introduction, we would like to mention \cite{CLZ10,JPZ12,PP12} for some related contributions
to orthogonally additive scalar holomorphic functions of bounded type on $C(K)$-spaces and on general C*-algebras.


\section{Preliminaries}\label{s:linear}

\subsection{Orthogonally additive and zero product preserving polynomials}

For a C*-algebra $A$, we write $A_{sa}$ and $A_+$ for the set  of all its self-adjoint and positive elements, respectively.
The following proposition follows from  \cite[Lemma 2.1]{BLW14} and its proof (see also \cite[Proposition 6]{GPPI13}).

\begin{proposition}\label{prop}
Let $A,B$ be C*-algebras.  Let
$H: B_{A}(0;r)\to B$ be a holomorphic function with Taylor series
$H = \sum_{n=0}^\infty P_n$ at zero.  Let $D= A$, $A_{sa}$ or $A_+$.
\begin{enumerate}[(a)]
\item If $H$ is orthogonally additive on $B_{E}(0;r)\cap D$ then each $P_n$ is orthogonally additive on $D$.
\item If $H$ is zero product preserving on $ B_{E}(0;r)\cap D$ then each $P_n$ is  zero product preserving on $D$.
Indeed, we have
$$
fg=0 \quad\implies\quad P_m(f)P_n(g) =0, \quad\forall f,g\in D, \forall m,n = 0,1,2,\ldots.
$$
\end{enumerate}
\end{proposition}
It follows from  Proposition \ref{prop}
that if $H$ is orthogonally additive or zero product preserving, then the constant term $P_0 = 0$.

The following linearization of orthogonally additive $n$-homogeneous
polynomials is the key tool of us.  Note that although the theorem in \cite{PV} is stated for compact spaces, the proof
there works also for locally compact spaces.

\begin{theorem}[{\cite[Theorem 2.1]{PV}; see also \cite{BLL06, CLZ06}}]\label{thm:PV}
 Let $X$ be a locally compact Hausdorff space and $F$ a Banach space.  Let $P: C_0(X)\to F$
 be a bounded orthogonally additive $n$-homogeneous polynomial. Then there exists a bounded linear
operator $T : C_0(X)\to F$ such that
$$
P(f)  = T(f^n),\quad \forall f\in C(K).
$$
\end{theorem}

As  commutative   C*-algebras are algebras of continuous functions, the following can be considered as the
non-commutative version of Theorem \ref{thm:PV}.

\begin{theorem}[\cite{PPV08,BFGP09}] \label{thm:C-form}
Let $A$ be a C*-algebra, $F$ a complex Banach space, and $P:A\to F$ a bounded
$n$-homogeneous polynomial. The following are equivalent.
\begin{enumerate}[(1)]
    \item There exists a bounded linear operator $T : A \to F$ such that
$$
P(a) = T(a^n), \quad\forall a\in A.
$$

    \item $P$ is orthogonally additive on $A$, i.e.,
    $$
    ab=ba=0\quad\implies P(a+b)=P(a) + P(b), \quad\forall a,b\in A.
    $$

    \item $P$ is orthogonally additive on  $A_{sa}$, i.e.,
    $$
    \quad ab=0\quad\implies P(a+b)= P(a) + P(b), \quad\forall a,b\in A_{sa}.
    $$
\end{enumerate}
\end{theorem}

In view of Theorem \ref{thm:C-form}, the assumption in Theorem \ref{thm:PV} can be weakened to
$P$ being orthogonally additive on $C_0(X)_{sa}$.
So due to Proposition \ref{prop} when we say a holomorphic map or a homogeneous polynomial on a C*-algebra to be orthogonally additive, it does not matter it is orthogonally additive on all elements or just on self-adjoint ones.

Similarly, we have

\begin{lemma}\label{lem:OM}
Let $A,B$ be C*-algebras.  Let $P:A\to B$ be a bounded orthogonally additive $n$-homogeneous polynomial.
Consider the following conditions.
\begin{enumerate}[(1)]
    \item $P$ is zero product preserving on $A$, i.e.,
     $$
    \quad ab=0\quad\implies P(a)P(b)=0, \quad\forall a,b\in A.
    $$
    \item $P$ is zero product preserving on $A_{sa}$, i.e.,
     $$
    \quad ab=0\quad\implies P(a)P(b)=0, \quad\forall a,b\in A_{sa}.
    $$
    \item $P$ is zero product preserving on $A_+$, i.e.,
     $$
    \quad ab=0\quad\implies P(a)P(b)=0, \quad\forall a,b\in A_+.
    $$
\end{enumerate}
We have (1) $\Rightarrow$  (2) $\Leftrightarrow$ (3) in general; and all
three conditions are equivalent when $A$ is commutative.
\end{lemma}
\begin{proof}
It is clear that (1) $\Rightarrow$  (2) $\Rightarrow$ (3).
Suppose  $a,b\in A_{sa}$ with $ab=0$.
Write $a=a_+-a_-$ and $b=b_+ - b_-$ as the orthogonal differences of their positive and negative parts.
By functional calculus, we see that
$$
a_+b_+ = a_+b_- = a_-b_+ = a_-b_- = 0.
$$
Now the orthogonally additivity of $P$ and condition (3) give
$$
P(a)P(b) = P(a_+)P(b_+) + P(a_+)P(-b_-) + P(-a_-)P(b_+)+ P(-a_-)P(-b_-) = 0.
$$
Hence, we have (3) $\Rightarrow$ (2).

Now suppose $A=C_0(X)$ is commutative.  We verify the implication
(3) $\Rightarrow$ (1).
Let $T:A\to B$ be the bounded linear operator associated to $P$ as in
Theorem \ref{thm:C-form}.

First let $f, g$ in $C_0(X)$ be real-valued such that $fg = 0$. Then
$$
f_+  g_+ = f_+ g_- = f_-  g_+ = f_-  g_- = 0,
$$
and hence,
$$
\sqrt[n]{f_+} \sqrt[n]{g_+} = \sqrt[n]{f_+}  \sqrt[n]{g_-}
=\sqrt[n]{f_-}  \sqrt[n]{g_+} =\sqrt[n]{f_-}  \sqrt[n]{g_-} = 0.
$$
It follows from Theorem \ref{thm:C-form}(1) and condition (3) that
\begin{eqnarray*}
& & T(f)  T(g) \\
 &=& T(f_+)  T(g_+) - T(f_+)  T(g_-) -\ T(f_-)  T(g_+) + T(f_-)  T(g_-)\\
 &=& P(\sqrt[n]{f_+})  P(\sqrt[n]{g_+}) - P(\sqrt[n]{f_+})  P(\sqrt[n]{g_-}) -\ P(\sqrt[n]{f_-})  P(\sqrt[n]{g_+}) + P(\sqrt[n]{f_-})  P(\sqrt[n]{g_-})\\
 &=& 0.
\end{eqnarray*}

Let $f_1, f_2, g_1, g_2$ in $C_0(X)$ be real-valued such that $(f_1 + ig_1)\cdot (f_2 + ig_2) = 0$. Then
for each $t$ in $X$, we have $f_1(t)=g_1(t)=0$ or $f_2(t)=g_2(t)=0$.
Therefore,
$$
f_1\cdot f_2 = g_1\cdot g_2 = f_1\cdot g_2 = g_1\cdot f_2 = 0.
$$
It follows from above that
\begin{eqnarray*}
 && T(f_1 + ig_1) \cdot T(f_2 + ig_2)\\
  &=& T(f_1)\cdot T(f_2) + iT(f_1)\cdot T(g_2)
  +\ iT(g_1)\cdot T(f_2) - T(g_1)\cdot T(g_2)
 = 0.
\end{eqnarray*}
Therefore, $T$ preserves zero products.  Let $f,g\in C_0(X)$ with $fg=0$.  Then
$f^ng^n=0$ as well.  Consequently,
$$
P(f)P(g)=T(f^n)T(g^n)=0,
$$
and  the assertion follows.
\end{proof}

We remark that the surjective linear isometry $a\mapsto a^t$ gives a counter example for the implication
(2) $\Rightarrow$ (1) in Lemma \ref{lem:OM},
when $A=B=B(\cH)$.  Here,
$a^t$ denotes the transpose of an operator $a$ in $B(\cH)$ with
respect to an arbitrary but fixed orthonormal basis of the Hilbert space $\cH$.

\subsection{Classical Banach-Stone type theorems}

The
classical Banach-Stone theorem states in the following two ways.

\begin{theorem}[{Banach-Stone Theorem for isometries; see, e.g., \cite{Conway90}}]\label{thm:BSTISO}
Let $T: C_0(X) \rightarrow C_0(Y)$ be a bounded
linear operator. If $T$ is a surjective isometry, then there
exists a homeomorphism $\varphi: Y\rightarrow X$ such that
$$
Tf(y) = h(y)f(\varphi(y)), \quad\forall\; f \in C_0(X), \;\forall\; y \in Y.
$$
Here, $h$ is a continuous unimodular scalar function on $Y$, i.e., $|h(y)| = 1, \forall y \in Y$.
\end{theorem}

\begin{theorem}[{Banach-Stone Theorem for zero product preserving
maps; see \cite{Ab83,Ar83}, and also \cite{Ja90,FH94, JW96}}]\label{thm:BSTRI}
 Let $T: C_0(X) \rightarrow C_0(Y)$ be a bounded
linear operator. If $T$ is  zero product preserving, then there exist a bounded scalar function
$h$ on $Y$ and a
map $\varphi: Y\rightarrow X$ such that
$$
Tf(y) = h(y)f(\varphi(y)), \quad \forall\; f \in C_0(X),\;\forall\; y \in Y.
$$
Both $\varphi$ and $h$ are continuous on the  cozero set $\coz(h)
:=\{y\in Y: h(y)\neq 0\}$, which is open in $Y$.
If $T$ is bijective, then $h$ is  away from zero and $\varphi$ is a homeomorphism from $Y$ onto $X$.
\end{theorem}

The original form of Theorem \ref{thm:BSTRI} states for Riesz
isomorphisms (see, e.g., \cite[p. 172]{MN}). Note that  a linear
operator on Riesz spaces is a Riesz homomorphism if and only if it
is positive and disjointness preserving.  So the above form is
indeed an improvement.  In the following, we will see that  the disjointness
structure ($=$ zero product structure) alone gives rise to a rather rich theory.

The classical Banach-Stone theorems have been generalized in several
contexts and appeared in, e.g., the  vector-valued version, the lattice version, and the
C*-algebra version \cite{Ab83,Ar83,Ja90,FH94,JW96,EO,CCW}.
Going into a different direction
from \cite{BL05},
we will continue this line and obtain polynomial versions for C*-algebras in Sections \ref{s:zp} and \ref{s:isometry}.


\section{Orthogonally additive and zero product preserving holomorphic maps} \label{s:zp}

Let $X$ and $Y$ be locally compact Hausdorff spaces.
Let   $P:C_0(X)\rightarrow C_0(Y)$ be a
bounded orthogonally additive
$n$-homogeneous polynomial.
By Theorem \ref{thm:PV}, there exists a
bounded linear operator $T:C_0(X)\rightarrow C_0(Y)$ such that
\begin{align}\label{eq:5}
P(f) = T(f^n), \quad \forall\; f \in C_0(X).
\end{align}

In the following, we call a subset $F$  of $C_0(Y)$
        \begin{itemize}
            \item  \emph{separating points in $Y$} \emph{strongly } (resp.\ \emph{strictly}) if
for every pair of distinct  points $y_1,y_2$ in $Y$ there is an $f$ in $F$ such that
$|f(y_1)|\neq|f(y_2)|$ (resp.\ $f(y_1)\neq f(y_2)=0$).
            \item \emph{regular} if for any closed subset $Y_0$ of $Y$ and any point $y$
in $Y\setminus Y_0$ there is an $f$ in $F$ such that $f=0$ on $Y_0$ and $f(y)\neq0$.
    \end{itemize}
We say that an $n$-homogeneous polynomial $P: C_0(X)\to C_0(Y)$
\emph{has trivial positive kernel} if
$$
P(f) =0 \ \implies\  f=0, \quad\forall f\in C_0(X)_+.
$$

\begin{theorem}[Banach-Stone Theorem for zero product preserving polynomials]\label{thm:BSTPdp}
Let $P: C_0(X)\rightarrow
C_0(Y)$ be a bounded orthogonally additive  $n$-homogeneous polynomial.
Assume that $P$ is zero product preserving on positive elements, i.e.,
$$
fg=0 \quad\text{$\implies$}\quad   P(f)P(g) = 0, \quad \forall f,g \in C_0(X)_+.
$$
Then there exist a bounded scalar function $h$ on $Y$ and a map
$\varphi: Y\rightarrow X$ such that
$$
P(f)(y) = h(y)(f(\varphi(y)))^n, \qquad \forall f \in C_0(X), \;\forall\;  y \in Y.
$$
Here, both $h$ and $\varphi$ are  continuous wherever $h$ is nonvanishing.

If, in addition, $P$ has trivial positive kernel
and its range separates points in $Y$ strictly,
then $h$ is nonvanishing on $Y$ and $\varphi$ is a homeomorphism from $Y$ onto a dense subset of $X$.
\end{theorem}
\begin{proof}
Let $T$ be the bounded linear map associated to $P$ as in \eqref{eq:5}.
As seen in the proof of Lemma \ref{lem:OM},
$T$ is   zero product preserving.
It follows from Theorem \ref{thm:BSTRI} that there exists a
bounded weight function $h$ on $Y$ continuous on its open cozero set
$Y_1=\{y\in Y: h(y)\neq 0\}$, and
a continuous map $\varphi: Y_1\to X$ such that
\begin{align}\label{eq:8}
P(f)(y) = \left\{
            \begin{array}{ll}
              h(y)f(\varphi(y))^n, & \forall f\in C_0(X), \forall y\in Y_1,\\
              0, & \text{ on } Y\setminus Y_1.
            \end{array}
          \right.
\end{align}

Without further assumptions, we might define $\varphi$ arbitrary on $Y\setminus Y_1$.
This finishes the first assertion of the theorem.

\textbf{Claim 1}.  If $P$ has trivial positive kernel, then $\varphi(Y_1)$ is dense in $X$.

For else, there would be a nonempty open set $U$ in $X$
disjoint from   $\varphi(Y_1)$.
Let $f$ be nonzero in $C_0(X)_+$ vanishing outside $U$.
Then $P(f)=0$ implies $f=0$, a contradiction.

\textbf{Claim 2}.  If $P$ separates points in $Y$
strictly, then $Y=Y_1$ and $\varphi$ is a homeomorphism  from $Y$ onto $\varphi(Y)$.

Since for every $y$ in $Y$, there is an $f$ in $C_0(X)$ such that $P(f)(y)\neq 0$,
we see that $Y=Y_1$ on which $h$ is nonvanishing.  Hence
$$
P(f)(y) = h(y)f(\varphi(y))^n, \quad \forall f\in C_0(X), \forall y\in Y.
$$
Moreover, for any distinct points $y_1,y_2$ in $Y$ the strong separation property of
$P$ ensures again that $\varphi(y_1)\neq \varphi(y_2)$.
Thus, $\varphi$ is one-to-one.
Finally, it is routine to see that  $\varphi$ is a homeomorphism from $Y$ onto $\varphi(X)$.
\end{proof}

\begin{example}
We remark that $\varphi(Y)$ can be a proper dense subset of $X$ in Theorem \ref{thm:BSTPdp}.
For example, consider the  map $P:C[0,1]\to C_0(0,1]$ defined by $P(f)(t)=tf(t)$.
On the other end, the weight function $h$ might not be  continuous on the whole $Y$.
The  map $P:C_0(0,1]\to C[0,1]$ defined by $P(f)(t)=\sin(1/t)f(t)$ on $Y_1=(0,1]$ verifies this fact.
Moreover, the strong separation assumption on the range of $P$ cannot be weakened to the usual one.
Consider $P:C[0,1] \to C([0,1]\cup[2,3])$ defined by $P(f)\mid_{[0,1]}(t)= f(t)$ and $P(f)\mid_{[2,3]}(t) = f(t-2)/2$.
It is obvious that the range of $P$ separates points in $Y=[0,1]\cup[2,3]$ (but not strictly), and $\varphi(t)=\varphi(2+t)$ for
all $t$ in $[0,1]$.
\end{example}

\begin{theorem}
\label{thm:BSTPdphf}
Let $H: B_{C_0(X)}(0;r)\to C_0(Y)$ be a
bounded orthogonally additive and zero product preserving holomorphic function.
Then there exist a  sequence $\{h_n\}$ of bounded scalar  functions on $Y$ in which each $h_n$ is continuous on its cozero set, which is open,
and a map
$\varphi: Y\rightarrow X$ such that
$$
H(f)(y) = \sum_{n=1}^\infty h_n(y)(f(\varphi(y)))^n, \quad\forall y\in Y,
$$
uniformly for all $f $ in $B_{C_0(X)}(0;r)$.
Here,  $\varphi$  is continuous  wherever some $h_n$ is nonvanishing.
\end{theorem}
\begin{proof}
The holomorphic function $H: B_{C_0(X)}(0;r)\to C_0(Y)$ has its Taylor series $\sum_{n\geq 0} P_n$ at zero.
By Proposition \ref{prop},
$P_n: C_0(X)\to C_0(Y)$ is a bounded orthogonally additive
and zero product preserving $n$-homogeneous polynomials for every $n \geq 1$, and $P_0 = 0$.
It follows from Theorem \ref{thm:BSTPdp}  that there exist for each $n=1,2,\ldots$, a
 bounded
scalar function $h_n$ on $Y$  continuous on its cozero set
$\coz(h_n):=\{y\in Y: h_n(y)\neq 0\}$, which is open, and a
map $\varphi_n: Y\to X$ continuous on
$\coz(h_n)$ such that
$$
P_n(f)(y) = h_n(y)f(\varphi_n(y))^n, \quad\forall\, f\in C_0(X), \forall\, y\in Y, \forall\, n =1,2,\ldots.
$$

For any two positive integers $m\neq n$,  we claim that
$$
\varphi_m(y) = \varphi_n(y), \quad\forall y\in \coz(h_m)\cap\coz(h_n).
$$
Suppose $y\in Y$ such that $h_m(y)h_n(y)\neq 0$ and $x_m=\varphi_m(y)\neq x_n = \varphi_n(y)$.
Let $f,g\in C_0(X)_+$ such that  $fg=0$, and $f(x_m) = g(x_n)=1$.
By Proposition \ref{prop}(b), we see that
$$
0=P_m(f)(y)P_n(g)(y)=h_m(y)h_n(y).
$$
This contradiction shows that $\varphi_m, \varphi_n$ agree on $\coz(h_m)\cap\coz(h_n)$.
Therefore, we can define a map $\varphi: Y\to X$ by set-theoretical union, which agrees with $\varphi_n$
and is continuous on $\coz(h_n)$ for $n=1,2,\ldots$.
\end{proof}



To obtain the non-commutative version of Theorem \ref{thm:BSTPdphf}, we need the following counterpart of Theorem \ref{thm:BSTPdp}.

\begin{theorem}\label{thm:C-dp}
Let $A, B$ be  C*-algebras.  Let $P:A\to B$ be a bounded orthogonally additive $n$-homogeneous polynomial.
Let $T: A\to B$ be the bounded linear operator such that $P(a)=T(a^n), \forall a\in A$.  Let
$h:= T^{**}(1)$, where $T^{**}$ is the bidual map of $T$.
Suppose that
$$
ab = 0 \quad\implies\quad P(a)P(b)=0,\quad \forall a,b \in A_+.
$$
\begin{enumerate}[(a)]
    \item If  $h$ is invertible, then there is a Jordan homomorphism $J: A\to B$
 such that
$$
P(a)=hJ(a)^n=J(a)^n h, \quad \forall a\in A.
$$
    \item If  $B=\operatorname{span} P(A)$, the linear span of the range of $P$, then $h$ is a central invertible multiplier of $B$ and $J$ in (a) is  surjective.
\end{enumerate}
In both cases, $J$ is injective if and only if $P$ has trivial positive kernel.
\end{theorem}
\begin{proof}
By Theorem \ref{thm:C-form}, we have a bounded linear operator $T:A\to B$
such that $P(a)=T(a^n)$, $\forall a\in A$.
By functional calculus,  for every $x$ in $A_{sa}$ there is a self-adjoint element $y$ in
$C^*(x)$, the C*-subalgebra of $A$ generated by $x$, such that ${y}^n=x$.
It is then easy to see that $T$ sends zero  products in $A_{sa}$ to zero products in $B$.
It follows from \cite[Lemma 4.5]{CKLW03} and \cite[Lemma 2.3]{W07} that $T^{**}(1)T(a^2) = T(a^2)T^{**}(1) =(T(a))^2$ for all $a$ in $A$.
If $h=T^{**}(1)$ is invertible,  we have the asserted Jordan homomorphism $J=h^{-1}T$.

Now, suppose $B=\operatorname{span} P(A)$ instead.
In this case, every element $b$ in $B$ can be written as a linear sum
$b=\sum_j \alpha_j P(b_j) = \sum_j \alpha_j T(b_j^n)=T(\sum_j \alpha_j b_j^n)$.
Therefore, $T$ is surjective.
Moreover, $T$ sends zero products in $A_{sa}$ to zero products in $B$.  By \cite[Theorem 2.4]{W07} (where $A$
can be non-unital; see also \cite[Theorem 4.12]{CKLW03} for the unital case),
we see that $h$ is an invertible central multiplier of $B$, and
$J=h^{-1}T$ is a bounded surjective Jordan homomorphism from $A$ onto $B$.

Finally, the kernel of $J$ is a closed two-sided ideal of $A$ (\cite{Civin65}).
In particular, it is generated by its positive elements.
For any positive element $a$ in $A$,
we have $P(a)=0$ if and only if $J(a)=0$.  Hence, $J$ is injective if and only if $P$ has trivial positive kernel.
\end{proof}

The following lemma might be known, although we have not found a reference from the literature.
We thank Lawrence G. Brown for telling us the following proof.

\begin{lemma}\label{lem:larry}
Let $A$ be a non-commutative C*-algebra.  Then there exist $a,b$ in $A$ such that
$ab=0$ but $b^n a^n\neq 0$ for $n=1,2,3,\ldots$.
\end{lemma}
\begin{proof}
By Kaplansky's
theorem \cite[page 292]{KR} there is a norm one element $x$ in $A$ with
$x^2 = 0$.  Let $h= |x|$ and $k = |x^*|$.
Consider the left support projection $p=\lim_n k^{1/n}$
and the right support projection $q=\lim_n h^{1/n}$ of $x$.  Then $p,q$
are open projections of $A$ such that $pq=0$ and $x= pxq$.
Moreover, $1$ is
in the spectrum of the positive norm one element $k$.
Let $a = k$ and $b = x^* + h$.  Then $ab=0$.

We verify that $b^n a^n\neq0$.
Let $B$ be the commutative C*-subalgebra of $A$ generated by the orthogonal positive elements $h$ and $k$.
Indeed,  $B$ consists of elements
$f(k) + g(h)$, where $f$ and $g$ are continuous functions vanishing at $0$.  Note that
every complex homomorphism of $B$ extends to a pure state of $A$.  Thus there is
a pure state $\phi$ such that $\phi(f(k)) = f(1)$, and $\phi(g(h)) = 0$.  Consider
the GNS representation $(\pi,\cH,v)$ for $\phi$, where $v$ is the state vector.  Thus $\pi(f(k))v
= f(1)v$, and $\pi(g(h))v = 0$.  Let $w = \pi(x)^*v$.  Then $v$ and $w$ form an orthonormal
basis for a two-dimensional subspace of the Hilbert space $\cH$ which is
invariant under both $\pi(x)$ and $\pi(x)^*$.  The matrix representations of the
restrictions of $\pi(a)$ and $\pi(b)$ to this subspace can be written as two
idempotent $2 \times 2$
matrices
$$
\left(
  \begin{array}{cc}
    1 & 0 \\
    0 & 0 \\
  \end{array}
\right)
\quad\text{and}\quad
\left(
  \begin{array}{cc}
    0 & 0 \\
    1 & 1 \\
  \end{array}
\right).
$$
Their product is $0$ in one order but non-zero in the other order.
In particular, $\pi(b^n a^n)= \pi(b)^n\pi(a)^n = \pi(b)\pi(a)\neq 0$.  Thus
$b^n a^n\neq 0$ for $n=1,2,3,\ldots$.
\end{proof}

\begin{theorem}\label{thm:C-dp-type1}
Let $A, B$ be  C*-algebras.
Let $P:A\to B$ be a bounded orthogonally additive $n$-homogeneous polynomial.
Suppose that  $B=\operatorname{span} P(A)$, and
$$
ab = 0 \quad\implies\quad P(a)P(b)=0,\quad \forall a,b \in A.
$$
Then there is a central invertible multiplier $h$ of $B$ and a bounded surjective
algebra homomorphism $J$ from $A$ onto $B$ such that
$$
P(a)=hJ(a)^n, \quad \forall a\in A.
$$
Moreover, $J$ is an algebra isomorphism if and only if $P$ has trivial positive kernel.
\end{theorem}
\begin{proof}
In view of Theorem \ref{thm:C-dp}, it suffices to verify that the surjective Jordan
homomorphism $J: A\to B$ is  multiplicative.
By  Bre\v{s}ar's theorem \cite[Theorem 2.3]{bresar89},
there are closed ideals $I_1,I_2$ of $A$ and $I'_1, I'_2$ of $B$ satisfying the following properties.
\begin{enumerate}[(i)]
    \item $I_1 + I_2$ is an essential ideal of $A$ with $I_1\cap I_2 = \ker J$.
    \item $I'_1 + I'_2$ is an essential ideal of $B$ with $I'_1\cap I'_2 = \{0\}$.
    \item  $J(I_1)=I'_1$ and $J(I_2)=I'_2$.
    \item $J(ux) = J(u)J(x), \ \forall u\in I_1, \forall x\in A$.
    \item $J(vx) = J(x)J(v), \ \forall v\in I_2, \forall x\in A$.
\end{enumerate}

Let $I$ be the kernel of the Jordan homomorphism
$J|_{I_2}$.  Then $I$ is a closed two-sided ideal of the C*-algebra $I_2$ (\cite{Civin65}).
Therefore, $J$ induces a Jordan isomorphism $\tilde{J}$ from the C*-algebra $I_2/I$ onto $I'_2$.
By (v), $\tilde{J}$ is anti-multiplicative.

We claim that $I'_2$ is commutative.  Otherwise, there will be $a,b$ in $I'_2$ such that
$a'b' =0$ but ${b'\,}^n {a'\,}^n\neq 0$ by Lemma \ref{lem:larry}.
Let $a,b\in  I_2/I$ such that  $\tilde{J}(a)=a'$ and $\tilde{J}(b)=b'$.
Then
$$
\tilde{J}(ba)=\tilde{J}(a)\tilde{J}(b)= a'b' = 0.
$$
Since $\tilde{J}$ is injective, $ba=0$ in $I_2/I$.
By \cite[Proposition 2.3]{akemann}
(see also \cite[Lemma 4.14]{CKLW03}), there are $c,d$ in $I_2$ such that $a=c+I$, $b=d+I$ and $dc=0$.
It follows from the zero product preserving property of $P$ that
$$
P(d)P(c) = h^2J(d)^nJ(c)^n = 0.
$$
Since $h$ is invertible, we have
$J(d)^nJ(c)^n=0$.  This in turn provides a contradiction that
$$
0 = \tilde{J}(b)^n \tilde{J}(a)^n = {b'\,}^n{a'\,}^n \neq 0,
$$
which verifies the commutativity of $I'_2$.

Denote by $W=I_1 + I_2$  the essential ideal  of $A$.  It follows from  (iv) and (v) that
$J(wx) = J(w)J(x)$ for all  $w$ in $W$ and $x$ in $A$.
Consequently, for all $w$ in $W$ and $x,y$ in $A$ it holds
$$
J(w)J(xy)= J(wxy) = J(wx)J(y)= J(w)J(x)J(y).
$$
It turns out that
$$
J(W)(J(xy)-J(x)J(y)) = 0.
$$
As $J(W)=I'_1 + I'_2$ is an essential ideal of $B$ by (ii), we establish the desired conclusion that
$J(xy)=J(x)J(y)$, and thus $J$ is an algebra homomorphism.
\end{proof}

Recall that a standard C*-algebra $A$ on a Hilbert space $\cH$ is a C*-subalgebra of $B(\cH)$ containing
all compact operators.  In particular, $B(\cH)$ and $K(\cH)$, the C*-algebra of compact operators, are standard.

\begin{corollary}\label{cor:BH}
Let $\cH$ be a complex Hilbert space of arbitrary dimension.
Let $A$ be a standard C*-algebra on $\cH$.
Let $P:A\to A$ be a bounded orthogonally additive $n$-homogeneous polynomial such that
$A=\operatorname{span} P(A)$.
\begin{enumerate}[(a)]

    \item If $P(a)P(b)=0$ whenever $a,b\in A_+$ with $ab=0$, then there exist a nonzero scalar $\lambda$ and
    an invertible operator $S$ in $B(\cH)$ such that either
    $$
    P(a) = \lambda Sa^nS^{-1} ,\ \forall a\in A\quad\text{or}\quad P(a) = \lambda S(a^t)^nS^{-1},\ \forall a\in A.
    $$

    \item If $P(a)P(b)=0$ whenever $a,b\in A$ with $ab=0$, then there exist a nonzero scalar $\lambda$ and
     an invertible operator $S$ in $B(\cH)$ such that
    $$
    P(a) = \lambda Sa^nS^{-1} ,\ \forall a\in A.
    $$
\end{enumerate}
\end{corollary}
\begin{proof}
By Theorems   \ref{thm:C-dp} and \ref{thm:C-dp-type1}, we  obtain a surjective
Jordan or algebra homomorphism $J:A\to A$.  Note that the kernel of $J$ is a two-sided ideal of $A$ (\cite{Civin65}).
It follows from \cite[Lemma 2]{KLW04} (and its proof) that $J$ is indeed bijective.
The assertions then follow from the known facts about
Jordan and algebra automorphism of standard C*-algebras and the triviality of the center of $A$ (see, e.g., \cite{Monlar00}, \cite{Monlar02},
\cite[Corollary 3.2]{Chernoff73} and \cite[\S 6]{Palmer94}).
\end{proof}

For holomorphic maps of matrices, we
have a counterpart to Theorem \ref{thm:BSTPdphf}.

\begin{theorem}[{\cite{BLW14}}]\label{thm:main-matrix}
Let $m$ and $s$ be
positive integers with $m\geq 2$ and $m\geq s$.
Let $H: B_{M_m}(0;r)\to M_s$ be a holomorphic function between  complex matrix algebras.
Assume $H$ is orthogonally additive and zero product preserving on self-adjoint elements.
Then either
\begin{enumerate}[(a)]
    \item the range of $H$ consists of zero trace elements
(this case occurs whenever $s<m$), or
    \item $s=m$, and
there exist a  scalar sequence $\{\lambda_n\}$ (some $\lambda_n$ can be zero) and an invertible $m\times m$
    matrix $S$ such that
\begin{align}\label{eq:ddag}
    H(x) = \sum_{n\geq 1} \lambda_n S^{-1}x^nS, \quad\forall x\in B_{M_m}(0;r),
\end{align}
or
    $$
    H(x) = \sum_{n\geq 1} \lambda_n S^{-1}(x^t)^nS, \quad\forall x\in B_{M_m}(0;r).
    $$
\end{enumerate}
In the  case (b), we always have the representation \eqref{eq:ddag} when $H$   preserves zero products, i.e.,
$$
ab =0 \quad\implies H(a)H(b) = 0, \qquad \forall a,b \in B_{M_m}(0;r).
$$
\end{theorem}

The following examples borrowed from {\cite{CKLW03,BLW14}} tell us that one cannot get a complete analog to Theorem \ref{thm:BSTPdphf} for the non-commutative case.
We also remark that Example \ref{ex:trival-multiplication}(c) below tells us that a similar conclusion of \cite[Theorem 18]{GPPI13}
for orthogonally additive and doubly orthogonality preserving holomorphic functions does not hold for zero product preserving ones.

\begin{examples}\label{ex:trival-multiplication}
Let $\{e_n\}$ be an orthonormal basis of a separable Hilbert space $\cH$.
Let $E_{ij}=e_i\otimes e_j$ be the matrix unit in $B(\cH)$ given by $E_{ij}(h) = \langle h, e_j\rangle e_i$.
\begin{enumerate}[(a)]
    \item
Consider the linear map
$\theta:B(\cH)\to B(\cH)$ defined by
$\theta(T):= E_{11}TE_{12}$.  Then $\theta$ is a bounded linear (and thus holomorphic) map.  Since
the range of $\theta$ has trivial multiplication, $\theta$ is zero product preserving.  However, $\theta$ cannot be written in the standard form
as stated in Corollary \ref{cor:BH} or Theorem \ref{thm:main-matrix}.
    \item
Consider $\theta:M_k\to M_{k+2}$ defined by
$$
\begin{pmatrix}a_{ij}\end{pmatrix}
\mapsto
\begin{pmatrix}
    0&a_{11}&a_{12}&\ldots&a_{1k}&0\\
    0&0&0&\ldots&0&a_{11}\\
    0&0&0&\ldots&0&a_{21}\\
    &\vdots&&\ddots&&\vdots\\
    0&0&0&\ldots&0&a_{k1}\\
        0&0&0&\ldots&0&0
    \end{pmatrix}.\quad
$$
Then $\theta$ is  linear (and thus holomorphic), and zero product preserving
on self-adjoint elements.  Note that
the range of $\theta$ does not have trivial multiplication, since
$\theta(E_{11})^2=E_{1,k+2}$.
However,  the range of $\theta$ consists of elements of zero trace.
We verify that $\theta$ cannot be written as the form $c\varphi$
for any fixed element $c$ in $M_{k+2}$ and any homomorphism or anti-homomorphism
$\varphi:M_k\to M_{k+2}$.
Assume, for example, that $\theta=c\vp$ and $\vp$ is a homomorphism.  Then we arrive at a contradiction
\begin{align*}
E_{1,k+2}&=\theta(E_{11})^2=\theta(E_{11})c\vp(E_{11})
    =\theta(E_{11})c(\vp(E_{12})\vp(E_{21}))\\
    &
    =\theta(E_{11})\theta(E_{12})\vp(E_{21})
    =0\vp(E_{21})=0.
\end{align*}

     \item Let $E$ and $F$ be the isometries in $B(\cH)$ such that $E(e_n)=e_{2n}$ and $F(e_n)=e_{2n-1}$ for $n=1,2,\ldots$, respectively.
Define a holomorphic function $\theta:B(\cH)\to B(\cH)$ by
$$
\theta(a) = EaE^* + Fa^2F^*, \quad\forall a\in B(\cH).
$$
Then $\theta$ is orthogonally additive and zero product preserving.
The range of $\theta$ contains the identity $\theta(1)=1$.  However, it cannot be written in any form as stated in Theorem \ref{thm:main-matrix}(b).
\end{enumerate}
\end{examples}

To get an analog result (Theorem \ref{thm:Conformal} below) to Theorem \ref{thm:BSTPdphf} for holomorphic maps between general C*-algebras, we need the following lemma.

\begin{lemma}\label{lem:OMH-alg}
Let $A,B$ be C*-algebras, $r>0$, and  $H=\sum_{n\geq1} P_n : B_A(0;r)\to B$ be an
orthogonally additive holomorphic map.
Suppose $H$ is zero product preserving on positive
(resp.\ all) elements in
$B_A(0;r)$.  Assume there is a polynomial term $P_k(x) = h_kJ(x)^k$
providing a  central invertible multiplier $h_k$ in $M(B)$ and a Jordan (resp.\ algebra) isomorphism $J:A\to B$.  Then
there are central multipliers $h_n$ in $M(B)$ for $n\geq1$ such that
$$
H(a) = \sum_{n\geq 1} h_nJ(a)^n,\quad\forall a\in B_A(0;r).
$$
\end{lemma}
\begin{proof}
Replacing $H$ with the map $x\mapsto J^{-1}(h_k^{-1}H(x))$, we can assume that $P_k(x)=x^k$ for all $x$ in $A=B$.
Let $T_n$ be the bounded linear map associated to $P_n$ such that $P_n(x)=T_n(x^n)$.
For any positive $x,y$ in $A_+$ with $xy=0$, by Proposition \ref{prop}(b) we have
$P_k(x)P_n(y) = P_n(y)P_k(x)=0$.  This gives $x^kT_n(y^n) =  T_n(y^n)x^k=0$.
It follows
\begin{align}\label{eq:local}
aT_n(b)= T_n(b)a =0\quad\text{whenever}\quad ab=0 \text{ and } a,b\in A_+.
\end{align}

Let $x,y\in M(A)_+$ with $xy=0$.  Choose $a_\lambda, b_\lambda\in A_+$ such that
$a_\lambda\uparrow x$ and $b_\lambda\uparrow y$ .  Since $a_\lambda b_\lambda=0$ for
all $\lambda$, \eqref{eq:local} and the $\sigma(A^{**}, A^*)$ continuity of $T_n^{**}$ give
\begin{align}\label{eq:local-M(A)}
xy=0 \quad\implies\quad xT_n^{**}(y)= T_n^{**}(y)x=0, \quad\forall x,y\in M(A)_+.
\end{align}

Let $a\in A_+$ with $\|a\|=1$.
 Identify the $C^*$-subalgebra
of $A$ generated by $1$ and $a$ with $C(X)$, where $X\subseteq
[0,1]$ is the spectrum of $a$.
Under this convention, $C(X)\subseteq M(A)$, and \eqref{eq:local-M(A)} applies.

Denote by $\theta:C(X)^{**}\to B^{**}$
the  map induced from $T_n^{**}$.
For
each positive integer $N$ and each integer $k=-1,0,1,\ldots, N$, let
$$
X_{N,k}= \left(\frac{k}{N},\frac{k+1}{N}\right]\cap X.
$$
Pick an arbitrary point $x_{N,k}$ from each nonempty $X_{N,k}$.
For any $f$ in $C(X)$,  we have
\begin{equation}\label{eq:f}
f=\lim_{N\to\infty} \sum_{X_{N,k}\neq\emptyset} f(x_{N,k})1_{X_{N,k}},
\end{equation}
where $1_{X_{N,k}}$ is the characteristic function of the Borel
set ${X_{N,k}}$, and the limit of the finite sums converges
uniformly on $X$. In particular, for every fixed positive integer
$N$ we have
$$
1=\sum_{X_{N,k}\neq\emptyset} 1_{X_{N,k}}.
$$
For two disjoint nonempty sets $X_{N,j}$ and $X_{N,k}$, we can
find two sequences $\{f_m\}_m$ and $\{g_m\}_m$ in $C(X)$ such that
$f_{m+p}g_m =0$ for $m,p=0,1,\ldots$, and $f_m \to 1_{X_{N,j}}$
and $g_m \to 1_{X_{N,k}}$ pointwisely on $X$.  By the weak
$^*$-continuity of $\theta$,  for all $m=1,2,\ldots$, we have
$$
1_{X_{N,j}}\theta(g_m) = \lim_{p\to\infty} f_{m+p}\theta(g_m)=0,
$$
and
$$
\theta(1_{X_{N,j}})g_m = \lim_{p\to\infty} \theta(f_{m+p})g_m=0.
$$
Thus
$$
1_{X_{N,j}}\theta(1_{X_{N,k}}) = \lim_{m\to\infty} 1_{X_{N,j}}\theta(g_m)=0,
$$
and
$$
\theta(1_{X_{N,j}})1_{X_{N,k}} = \lim_{m\to\infty} \theta(1_{X_{N,j}})g_m=0.
$$
Consequently, for each positive integer $N$ and each integer $j=-1,0,1,\ldots,N$,
we have
\begin{equation}\label{eq:h}
\theta(1_{X_{N,j}}) = \sum_{X_{N,k}\neq\emptyset} 1_{X_{N,k}}\theta(1_{X_{N,j}})
= 1_{X_{N,j}}\theta(1_{X_{N,j}}) = 1_{X_{N,j}}\theta(1)
= \theta(1) 1_{X_{N,j}}.
\end{equation}
It follows from \eqref{eq:f} and \eqref{eq:h} that
$$
\theta(f)=f\theta(1)=\theta(1)f, \quad\forall f\in C(X)_+.
$$
In particular, we have
$$
T_n(a)=aT_n^{**}(1)=T_n^{**}(1)a
$$
holds for all positive norm one, and thus all, elements $a$ in $A$.

Set $h'_n=T_n^{**}(1)$  for all $n=1,2,\ldots$.
We have thus obtained  a sequence $\{h'_n\}$ of central multipliers in $M(A)$ such that
$$
H(a) = \sum_{n\geq1} P_n(a) = \sum_{n\geq1} T_n(a^n) = \sum_{n\geq1} h'_n a^n, \quad\forall b\in A.
$$
Note that $h'_k=1$.

Going back to the original setting, we set $h_n= h_kJ(h'_n)$.
Since the surjective Jordan isomorphism $J$ sends central multipliers to central multipliers
\cite[P.~330]{Kadison51},
all $h_n$ are central multipliers in $M(B)$.  Moreover, we have
$$
J(h'_n a^n) = J(h'_na^n + a^n h'_n)/2 = J(h'_n)J(a^n), \quad\forall a\in A, n=1,2,\ldots.
$$
Consequently.
$$
H(a) = h_k J( \sum_{n\geq 1} h'_n a^n) = \sum_{n\geq 1} h_n J(a)^n, \quad\forall a\in A.
$$
\end{proof}

Recall that a holomorphic map $H$ is {conformal} (at $0$) if its derivative $P_1$ (at $0$) is a bounded invertible linear operator.
Combining Theorems  \ref{thm:C-dp} and \ref{thm:C-dp-type1}, and Lemma \ref{lem:OMH-alg}, we have the following main result of this section.

\begin{theorem}\label{thm:Conformal}
Let $A,B$ be C*-algebras, $r>0$, and  $H: B_A(0;r)\to B$ be an orthogonally additive
conformal holomorphic map.
Suppose $H$ is zero product preserving on positive
(resp. all) elements in
$B_A(0;r)$. Then there exist a sequence $\{h_n\}$ of central multiplier  in $M(B)$ and a Jordan (resp.\ algebra) isomorphism $J:A\to B$ such that
\begin{align}\label{eq:Conf-J}
H(a) = \sum_{n\geq 1} h_nJ(a)^n,\quad\forall a\in B_A(0;r).
\end{align}
\end{theorem}

The following result supplements Theorem \ref{thm:main-matrix}.  Note however that without the conformal assumption we might not have a positive result, as Example \ref{ex:trival-multiplication} demonstrates.

\begin{corollary}\label{cor:BHConformal}
Let $A$ and  $B$ be  standard C*-algebras on Hilbert space $\cH_1$ and $\cH_2$, respectively.
Let $H: B_{A}(0;r)\to B$ be an orthogonally additive conformal holomorphic map.
Suppose $H$ is zero product preserving on positive elements.
Then there exist a sequence $\{\lambda_n\}$ of scalars and  an invertible operator $S: \cH_2\to \cH_1$ such that either
\begin{align}\label{eq:BH-zpC}
    H(x) = \sum_{n\geq 1} \lambda_n S^{-1}x^nS, \quad\forall x\in B_{A}(0;r),
\end{align}
or
    $$
    H(x) = \sum_{n\geq 1} \lambda_n S^{-1}(x^t)^nS, \quad\forall x\in B_{A}(0;r).
    $$
If $H$ is zero product preserving on all elements in $B_{A}(0;r)$, then exactly the case
\eqref{eq:BH-zpC} holds.
\end{corollary}
\begin{proof}
It follows from Theorem \ref{thm:Conformal} that $H$ carries a form as in \eqref{eq:Conf-J}.
Since $M(B)$ has trivial center, all $\lambda_n :=h_n$ are scalars.
It is well-known that, see for example \cite{Monlar00, Monlar02}, the Jordan isomorphism $J: A\to B$ carries a form of either
$$
Jx = S^{-1}xS, \ \forall x\in A, \quad\text{or}\quad Jx=S^{-1}x^tS, \ \forall x\in A,
$$
where $S$ is a bounded invertible operator from $\cH_2$ onto $\cH_1$.
When $H$ preserves zero products on the whole of $B_{A}(0;r)$, the map $J$ is an algebra isomorphism.  Hence, exactly the case \eqref{eq:BH-zpC} happens.
\end{proof}

\section{Orthogonally additive and isometric polynomials}\label{s:isometry}

\begin{lemma}\label{lem:bs-isom}
Let $X$ and $Y$ be locally compact Hausdorff spaces.  Let $S:C_0(X)\to C_0(Y)$ be a
 linear map preserving   norms
of positive functions, i.e.,
$$
\|Sf\|=\|f\|, \quad\forall\; f\in C_0(X)_+.
$$
Suppose that the range of $S$ strongly  separates points in $Y$.
Then there exist
a homeomorphism $\psi$ from $X$ onto $\psi(X)\subseteq Y$, and a continuous unimodular scalar function $k$ on $X$ such that
$$
Sf(\psi(x))=k(x)f(x),\quad \forall\; f\in C_0(X), \;\forall\; x\in X.
$$
If the range of $S$ is regular,  then $S$ is a surjective linear isometry and $\psi(X)=Y$.
\end{lemma}
\begin{proof}
Let  $X_\infty = X\cup\{\infty\}$
be the   one-point
compactification of $X$.
In the case $X$ is compact, the point $\infty$ at infinity will be isolated in $X$.
We identify
$$
C_0(X) = \{f\in C(X_\infty): f(\infty)=0\}.
$$
The same also applies to $Y$ and $C_0(Y)$.

For every point $x$ in $X$, set
$$
S_x := \{y\in Y_\infty : |Sf(y)|=1 \text{ for all $f$ in $C_0(X)_+$  with } f(x) = \|f\| = 1\}.
$$
Clearly, $S_x$ is a compact subset of $Y$.

First, we verify that $S_x$ is  nonempty.  Otherwise, for every point $y$ in $Y_\infty$ there is an $f_y$ in
$C_0(X)_+$ with $f_y(x)=\|f_y\|=1$, but $|Sf_y(y)|<1$.  Let
$$
V_y =\{z\in Y_\infty: |Sf_y(z)| < 1\}.
$$
Then $V_y$ is an open neighborhood of $y$ containing the point $\infty$ at infinity.
The open covering $Y_\infty = \cup_y V_y$ of the compact space $Y_\infty$ has a finite subcover
$$
Y_\infty = V_{y_1}\cup \cdots\cup V_{y_n}.
$$
Let
$$
f=\frac{1}{n}(f_{y_1} + \cdots + f_{y_n}).
$$
Clearly, $f\in C_0(X)_+$ with $f(x) = \|f\| = 1$,
and
$$
|Sf(y)| \leq \frac{1}{n}(|Sf_{y_1}(y)| + \cdots + |Sf_{y_n}(y)|) < 1, \quad\forall y\in Y_\infty.
$$
This forces $1=\|f\|=\|Sf\|<1$, a contradiction.

Next, we verify that $S_x$ contains exactly a single point in $Y$.
Otherwise, let $y_1, y_2$ be two distinct points in $S_x$.
In other words,
\begin{align}\label{eq:twopoints}
|Sf(y_1)| = |Sf(y_2)| = 1 \quad\text{whenever}\quad f\in C_0(X)_+ \text{ with } f(x)=\|f\| =1.
\end{align}
Let $g$ be in $C_0(X)_+$ with norm one vanishing in a neighborhood of $x$.
Let $f$ be any function in $C_0(X)_+$ with $fg=0$ and $f(x)=\|f\|=1$.
If follows from \eqref{eq:twopoints} that
$$
|S(f+tg)(y_1)| = |S(f+tg)(y_2)| = 1, \quad\forall t\in [0,1].
$$
This forces
$Sg(y_1) = Sg(y_2)=0$.
Consequently, dealing separately with the positive and negative parts of the real and imaginary parts of
a continuous function we have
$$
Sg(y_1) = Sg(y_2) = 0\quad\text{whenever } g \text{ in } C_0(X) \text{ vanishes in a neighborhood of } x.
$$
Utilizing Uryshon's Lemma and the boundedness of $S$, we have indeed
\begin{align}\label{eq:twopointzero}
Sg(y_1) = Sg(y_2) = 0\quad\text{whenever } g \text{ in } C_0(X) \text{ vanishes at } x.
\end{align}
Therefore, there are scalars $\lambda_1, \lambda_2$ such that
$$
Sg(y_i) = \lambda_i g(x), \quad\forall g\in C_0(X),\ i=1,2.
$$
By \eqref{eq:twopoints}, we have $|\lambda_1|=|\lambda_2|=1$, and thus
$|Sg(y_1)|=|Sg(y_2)|$ for all $g$ in $C_0(X)$.  This is absurd, since the range of $S$ strongly  separates points in $Y$.

Now, we can define a function $\psi: X\to Y$ such that $S_x =\{\psi(x)\}$.
As in deriving \eqref{eq:twopointzero}, we have
$$
f(x) = 0 \quad\implies\quad Sf(\psi(x))=0, \qquad \forall f\in C_0(X).
$$
This provides a scalar $k(x)$ such that
\begin{align}\label{eq:wco}
Sf(\psi(x)) = k(x)f(x), \quad\forall f\in C_0(X), \forall x\in X.
\end{align}
It follows from the definition of $S_x$ that $|k(x)|=1$ for all $x$
in $X$.
Consequently, $\psi$ is one-to-one on $X$.

We claim that $\psi$ is a homeomorphism from $X$ onto $\psi(X)$.
To this end, suppose $x_\lambda \to x$ in $X$ and
$y$ is any cluster point of $\psi(x_\lambda)$ in $Y_\infty$.
It follows from \eqref{eq:wco} that $|Sf(y)|= |f(x)|=|Sf(\psi(x))|$ for all $f$ in $C_0(X)$.
By the strict separability assumption, $y=\psi(x)$, and thus $\lim_\lambda \psi(x_\lambda)=\psi(x)$.
Conversely, assume that $\psi(x_\lambda)\to \psi(x)$ and $\{x_\lambda\}$ has a cluster point $z$ in $X_\infty$.
By \eqref{eq:wco} again,
$|f(x)|=|f(z)|$ for all $f$ in $C_0(X)$.  This forces $z=x$, and $x=\lim_\lambda x_\lambda$ in $X$.

It is now plain that  $k$ is continuous on $X$.

Finally, if the range of $S$ is regular then \eqref{eq:wco} ensures that $\psi(X)$ is dense in $Y$.
Consequently, $S$ is a surjective linear isometry.
Applying what we have obtained to the inverse $S^{-1}: C_0(Y)\to C_0(X)$, we can verify that $\psi$ is
invertible and especially $\psi(X)=Y$.
\end{proof}

\begin{theorem}[Banach-Stone Theorem for $n$-isometries]\label{thm:BSTPiso}
Let $P: C_0(X)\rightarrow C_0(Y)$ be an orthogonally additive
$n$-homogeneous polynomial. Assume that $P$ is an $n$-isometry on positive elements, and its range is regular.
Then there exist a continuous unimodular scalar function $h$
and a homeomorphism $\varphi: Y\rightarrow X$ such that
$$
P(f)(y) = h(y)(f(\varphi(y)))^n, \qquad \forall\; f \in C_0(X), \;\forall\; y \in Y.
$$
\end{theorem}
\begin{proof}
Let $T:C_0(X)\to C_0(Y)$ be the bounded linear map associated to $P$ such that $P(f)=T(f^n)$ as in \eqref{eq:5}.
Since $P:C_0(X)\rightarrow C_0(Y)$ is an $n$-isometry on positive elements,  for every non-negative function $f$ in
$C_0(X)_+$ we have
$$
\|T(f)\| = \|P(\sqrt[n]{f})\| = \|\sqrt[n]{f}\|^n = \|f\|.
$$
Containing the regular subset $P(C_0(X))$ of  $C_0(Y)$, the range of $T$ is also  regular.
Lemma \ref{lem:bs-isom} applies and  yields
a homeomorphism $\psi$ from $X$ onto $Y$, and a continuous unimodular scalar function $k$ on $X$ such that
$$
Tf(\psi(x))=k(x)f(x),\quad \forall\; f\in C_0(X), \;\forall\; x\in X.
$$
Letting $\varphi:=\psi^{-1}$ and $h:=k\circ\varphi$, we arrive at the desired assertions.
\end{proof}

Without a tool  similar to Lemma \ref{lem:bs-isom},
we need extra assumptions in developing the following counterpart of Theorem \ref{thm:BSTPiso}.

\begin{theorem}\label{thm:C-isom}
Let $A, B$ be unital C*-algebras.  Let $P:A\to B$ be an orthogonally additive $n$-homogeneous polynomial.
Suppose that
$h:=P(1)$ is a unitary, $B=\operatorname{span} P(A)$, and $\|P(x)\| = \|x\|^n$ for every normal element $x$ in $A$.  Then there is a Jordan
$*$-isomorphism $J:A\to B$ such that
$$
P(a) = hJ(a)^n, \quad\forall a\in A.
$$
\end{theorem}
\begin{proof}
As in the proof of Theorem \ref{thm:C-dp}, we have a bounded surjective linear operator $T:A\to B$
such that $P(a)=T(a^n)$, $\forall a\in A$.
Replacing $P$ with $h^*P$, we can assume $T(1)=1$.
We are going to show that $T$ is a Jordan $*$-isomorphism.

We compute the norm of $T$ with the following formula (see, e.g., \cite[Theorem 2.14.5]{Li92})
\begin{align*}
\|T\|
&= \sup\{\|T(e^{ih})\| : h^*=h\in A\}\\
&= \sup\{\|P(e^{ih/n})\| : h^*=h\in A\}= 1.
\end{align*}
Let $x$ be a normal element in $A$.  Let $\alpha$ be a point in the spectrum of $x$ such that $\|x\|=|\alpha|$.
By functional calculus, we can find a normal element $y$ in $A$ such that $y^n=2\alpha+x$.
Observe that
$$
\|2\alpha + T(x)\| = \|P(y)\|  =\|y\|^n = \|2\alpha + x\| = 3\|x\|.
$$
Hence, $\|x\|\leq \|Tx\|\leq \|x\|$.  In other words, $T$ preserves the norm of every normal
element in $A$.  By \cite[Lemma 8]{Kadison51}, $T(a^*)=T(a)^*$ for every $a$ in $A$.
Consequently, $T$ induces a surjective unital linear isometry $T_{sa}$ between the JB-algebras $A_{sa}$
and $B_{sa}$.
It follows from \cite[Theorem 4]{WY78} that $T_{sa}$ is a Jordan isomorphism from  $A_{sa}$
onto $B_{sa}$, and thus
$T$ is a Jordan $*$-isomorphism from $A$ onto $B$.
\end{proof}

The following is a consequence of Theorem  \ref{thm:C-isom}, and the well-known facts about the structure of Jordan *-isomorphisms between
standard C*-algebras (see, e.g., \cite{Monlar00, Monlar02}).

\begin{corollary}\label{cor:isom-BH}
Let $\cH$ be a complex Hilbert space of arbitrary dimension.
Let $A$ be a unital standard C*-algebra on $\cH$.
Let $P:A\to A$ be a bounded orthogonally additive $n$-homogeneous polynomial such that
$A=\operatorname{span} P(A)$.
Suppose $h:=P(1)$ is a unitary and $\|P(a)\|=\|a\|^n$ for every normal operator $a$ in $A$.
Then there are unitary operators $U,V$
    in $B(\cH)$ such that either
    $$
    P(a) = Ua^nV,\ \forall a\in A\quad\text{or}\quad P(a) = U(a^t)^nV,\ \forall a\in A.
    $$
\end{corollary}

\end{document}